\documentclass[11pt]{amsart}
\usepackage[english]{babel}
\usepackage{amsfonts,amsthm,latexsym,mathtools,amsmath,amssymb,enumitem,appendix,color,hyperref}
\usepackage{graphicx} 
\usepackage{comment}

\newtheorem{theorem}{Theorem}
\newtheorem{lemma}{Lemma}
\newtheorem{proposition}{Proposition}

\theoremstyle{definition} 
\newtheorem{definition}{Definition}[section]

\theoremstyle{remark}

\newtheorem{remark}{Remark}

\newcommand{\floor}[1]{\lfloor #1 \rfloor}

\DeclareMathOperator{\rcov}{\mathrm{r_{cov}}}

\begin{document}

\title{Billiard Lawn Mowers}

\author[N.~Sriprasert]{Natnaree Sriprasert}
\address[Natnaree Sriprasert]{}
\email{natnaree.sri@kkumail.com}

\author[S.~Warakkagun]{Sangsan Warakkagun}
\address[Sangsan Warakkagun]{Department of Mathematics, Khon Kaen University, Khon Kaen, Thailand}
\email{sangwa@kku.ac.th}

\date{\today}

\begin{abstract}
We study the Lawn Mowing Problem restricted to periodic billiard paths in the unit square. Given the combinatorial data of a trajectory, we determine the optimal covering radius, and identify the shortest path that covers the square for any fixed blade radius.
\end{abstract}

\maketitle

\section{Introduction}
Imagine a circular lawn mower that travels in a straight line across a square yard, bouncing off the boundary with conservative reflection law like a billiard ball. For which such trajectories does the mower cover every point of the yard? Or given a fixed blade radius, what is the shortest trajectory guaranteed to mow the entire square?

These questions are inspired by the Lawn Mowing Problem of Arkin, Fekete, and Mitchell \cite{AFM}, in which the task is to find a shortest path whose neighborhood of fixed width covers a given planar region. 

The Lawn Mowing Problem is NP-hard in general. Even in a  $2 \times 2$ square, an \textit{optimal} route cannot have coordinates expressible in radicals \cite{FKPRS}. Restricting to billiard-like paths changes the picture. A trajectory starting off with irrational slope is dense, so in principle, such a path sweeps the whole square, but only after infinitely many bounces. We cannot wait forever. We therefore focus on periodic trajectories: those that return to their starting position after finitely many reflections. While these paths may not be the shortest since they generically self-intersect, they are practical, requiring little programmed input for they are completely determined by their initial configuration. Periodic trajectories have rational slopes and, as we will show, admit a clean geometric description. For such paths, we will solve our restricted Lawn Mowing Problem.

\section{The setup}
Throughout the paper, we think of the lawn as the unit square $\mathbf{T}=[0,1] \times [0,1]$, considered as a subset of $\mathbb{R}^2$. In our model, the center of the robotic mower is viewed as a point mass traveling along a \textit{trajectory}, a sequence of directed straight segments in $\mathbf{T}$, where each segment begins where the previous one ended, the angle of incidence equals the angle of reflection at each bounce, and the trajectory terminates if it meets a corner.

In this paper, we are interested in answering the following questions.
\begin{enumerate}[label=\textbf{Q\arabic*}]
\item \label{prob1}  What is the \emph{optimal covering radius} of a trajectory $\gamma$ 
\begin{equation*}
    \rcov(\gamma) := \inf\{ r>0: N_{r}(\gamma) \supset \mathbf{T} \},
\end{equation*}
that is, the minimal blade radius needed to cover $\mathbf{T}$ if $\gamma$ is followed indefinitely? Here, $N_r(\gamma)$ denotes the open $r$-neighborhood of $\gamma$. 
\item \label{prob2} Given a fixed radius $r$, what is the shortest billiard path $\gamma$ in $\mathbf{T}$ such that $N_{r}(\gamma) \supset \mathbf{T}$? 
\end{enumerate}

A trajectory is \emph{singular} if it meets some corner of $\mathbf{T}$. A non-singular trajectory $\gamma$ is said to be \emph{periodic} if the point returns to its starting point in the same initial position and direction after a finite number of bouncing points along the sides. 

Unless stated otherwise, periodic trajectories are \textit{primitive}, that is, they are traversed only once. The \textit{period} of a periodic trajectory is the number of bounces it makes before returning to the starting position, including the final one.

The \textit{orbit} of a trajectory $\gamma$ is the subset of $\mathbf{T}$ consisting of the points lying on $\gamma$. For brevity of notation, the orbit of $\gamma$ will also be denoted by $\gamma$, as the context will make it clear what we mean.

The remainder of this paper will give complete answers to \ref{prob1} and \ref{prob2} for periodic orbits in $\mathbf{T}$.

\section{Visualizing Periodic Orbits}

It is a well-known fact that non-singular trajectories in $\mathbf{T}$ with irrational slope do not terminate and have dense orbits in $\mathbf{T}$, see \cite{Tabachnikov}. For periodic orbits, we will give a complete description of their structure using the information of initial point on $\partial \mathbf{T}$ and direction. We summarize relevant facts from \cite{ChenOsinga} as needed; see the same reference for details.

\begin{definition}
    For $a \in (0,1)$ and relatively prime integers $p,q$, we let $\gamma(a,\frac{p}{q})$ be the trajectory beginning at the point $(a,0)$ on the side $[0,1]\times\{0\}$ of $\mathbf{T}$ with the initial angle $\arctan(\frac{p}{q})$, measured counterclockwise with respect to the positive $x$-axis. 
\end{definition}

\begin{theorem}
\label{thm: classification}\cite{ChenOsinga} The following statements hold for non-singular periodic trajectories in $\mathbf{T}$. 
\begin{enumerate}
    \item Every periodic trajectory in $\mathbf{T}$ has even period.
    \item Trajectories in $\mathbf{T}$ with period 2 are precisely the vertical and horizontal segments between opposite sides of $\mathbf{T}$. 
    \item  Every periodic trajectory of period at least 4 meets all four sides of $\mathbf{T}$, and the orbit of such trajectory coincides with that of some  $\gamma(a, \tfrac{p}{q})$, where $0 < a < 1$, $a \not \in \{ \frac{1}{p}, \ldots, \frac{p-1}{p}\}$, and $p,q$ are nonzero relatively prime positive integers.
\item The trajectory $\gamma(a, \frac{p}{q})$ is singular if and only if $a \in \left\{\tfrac{1}{p},\, \tfrac{2}{p},\, \ldots,\, \tfrac{p-1}{p}\right\}$. It is periodic otherwise.

\end{enumerate}
\end{theorem}

\begin{theorem}\label{thm: structure}
For $a \in (0,1)$ and relatively prime positive integers $p,q$, if the trajectory $\gamma(a,\frac{p}{q})$ is periodic in $\mathbf{T}$, then it has period $2(p+q)$, bouncing on each vertical side $p$ times and each horizontal side $q$ times. 
\end{theorem}

After drawing some periodic orbits, the reflective symmetries of $\gamma(a,\frac{p}{q})$ with respect to the lines $x = \frac{i}{p}$ and $y = \frac{j}{q}$  become evident (see Figure \ref{fig:orbit}). This fact has not been addressed in earlier works on periodic orbits \cite{ChenOsinga, Dulio} and we cannot find a reference, so we include a proof here.

\subsection{Behaviors of bouncing points}

From now on, for any real number $a$, we use the standard notation $ \{a\}= a -\lfloor a \rfloor$ be its decimal part.

Let $a \in (0,1)$ and $p,q$ be relatively prime positive integers. By unfolding a periodic orbit $\gamma(a,\frac{p}{q})$ in $\mathbb{R}^2$, that is, reflecting the square about the edge of a bouncing point every time the trajectory hits a side so that it continues in a straight line, we get the ray $\ell(a,\frac{p}{q}): y=\frac{p}{q}(x-a), ~y\geq 0$. When we project $\ell = \ell(a,\frac{p}{q})$ under the map $\pi:\ell \to \mathbf{T}$ defined by
$$\pi(x,y) = \begin{cases}
    (\{x\}, \{y\}), & \text{both $\floor{x}, \floor{y}$ are  even}\\
    (\{x\}, 1-\{y\}), & \text{$\floor{x}$ is even, $\floor{y}$ is odd}\\
    (1-\{x\}, \{y\}), & \text{$\floor{x}$ is odd, $\floor{y}$ is even}\\
    (1-\{x\}, 1-\{y\}), & \text{both $\floor{x}, \floor{y}$ are  odd},\\
\end{cases}$$
the image $\pi(\ell)$ is precisely the orbit of $\gamma(a,\frac{p}{q})$ in $\mathbf{T}$.

\begin{figure}
    \centering
\includegraphics[width=0.5\linewidth]{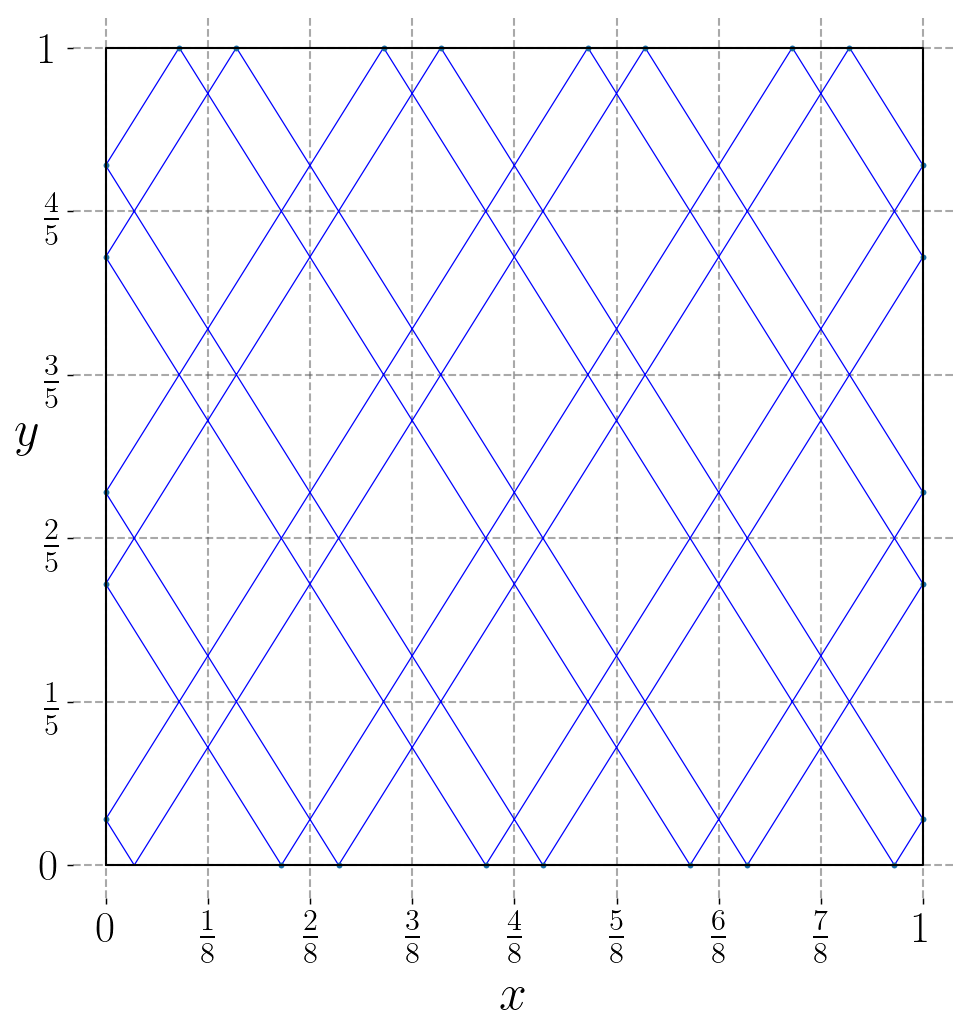} \quad\includegraphics[width=0.45\linewidth]{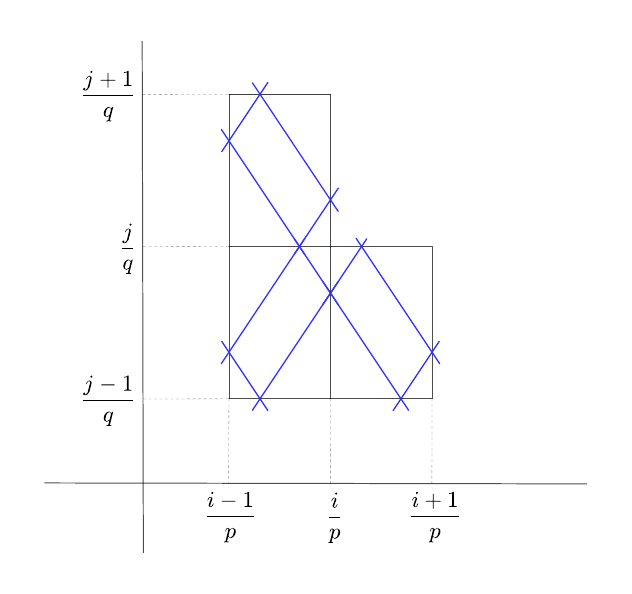}

    \caption{The orbit of $\gamma(0.02, 8/5)$ (left) and the schematic showing reflective symmetries along the grid lines of $W_{i,j}$ (right).} 
    \label{fig:orbit}
\end{figure}

The main result of this section is to show that the orbit $\gamma(a, \frac{p}{q})$ is symmetric in the following sense.

\begin{proposition}[Orbit symmetry]\label{prop: symmetry}
    Given the set up above, let $W_{i,j} = \left[ \frac{i-1}{p}, \frac{i}{p} \right] \times \left[ \frac{j-1}{q} , \frac{j}{q}\right] $ be a subrectangle of $\mathbf{T}$ for every $i \in \{1,\ldots, p\}$ and $j \in \{1,\ldots, q\}$. Then,  $\pi(\ell)\cap W_{i,j}$ is a parallelogram for any $i,j$. Moreover,
    \begin{enumerate}
        \item $\pi(\ell)\cap W_{i,j}$ and $\pi(\ell)\cap W_{i+1,j}$ are symmetric about the line $x =\frac{i}{p}$.
        \item $\pi(\ell)\cap W_{i,j}$ and $\pi(\ell)\cap W_{i,j+1}$ are symmetric about the line $y =\frac{j}{q}$.
    \end{enumerate}

\end{proposition}

We will prove Proposition \ref{prop: symmetry} via a series of lemmas. First, we will describe the structure of the bouncing points along the bottom side of $\mathbf{T}$. 

Suppose $\gamma(a,\frac{p}{q})$ is periodic. For $k \in \{0,1,\ldots, p-1\}$, define $x_k = 2k \frac{q}{p}+a$. By unfolding $\gamma(a,\frac{p}{q})$ in $\mathbb{R}^2$, we have that the points $\pi (x_k,2k)$, where $k\in \{0,1,\ldots, p-1\}$ are precisely the bouncing points along the bottom side of $\mathbf{T}$. Since $\gamma$ is non-singular, these bouncing points are not in the set $\left\{ \left(\frac{k}{p}, 0\right): k=1,\ldots, p-1 \right\}$ by Theorem \ref{thm: classification}.

\begin{lemma}\label{lem: different}
The points $\pi (x_k,2k)$, where $k\in \{0,1,\ldots, p-1\}$, are all distinct bouncing points on the bottom side of $\mathbf{T}$.
\end{lemma}

\begin{proof} 
The second coordinates clearly project to $0$, so we will only check the image of the first coordinate. First, suppose $\floor{x_k} \equiv \floor{x_l} \pmod{2}$ and $\{x_k\} = \{x_l\}$, where $k,l \in \{0,1,\ldots, p-1\}$. Then, $$0 = p\cdot|\{x_k\} -\{x_l\}| = |2(k-l)q - p(\floor{x_k}-\floor{x_l})|.$$ Since $\floor{x_k}-\floor{x_l}$ is an even integer, $p$ divides $(k-l)q$. Since $\gcd(p,q)=1$, we must have $k = l$. 

Suppose instead $\floor{x_k} \not \equiv \floor{x_l} \pmod{2}$ and $\{x_k\} = 1-\{x_l\}$ for some $k,l \in \{0,1,\ldots, p-1\}$. Then,
\begin{align*}
    0&= p\cdot|\{x_k\} - (1-\{x_l\})| \\
    &= |p(\{x_k\}+\{x_\ell\})-p| \\
    &= |(k+l)2q - p(\floor{x_k}+\floor{x_l}) -p +2a| \\
    &= |2(k+l)q - p(\floor{x_k}+\floor{x_l}+1)+2a|,
\end{align*}
forcing $2a$ to be an even integer. But $0 < a< 1$, so $2a$ cannot be an even integer, a contradiction. 

Thus, $\pi(x_k,2k)$ for $k=0,1,\ldots, p-1$, are all distinct.
\end{proof}

\begin{lemma}\label{lem: one-in-sub}
There is exactly one bouncing point in the subinterval $\left( \frac{i}{p}, \frac{i+1}{p} \right) \times \{0\}$ of the bottom edge.
\end{lemma}
\begin{proof}
For contradiction, suppose that there are two distinct bouncing points in $\left( \frac{i}{p}, \frac{i+1}{p} \right) \times \{0\}$. If $\floor{x_k} \equiv \floor{x_l} \pmod{2}$ with
$\{x_k\} = \frac{i}{p} + \epsilon_k$ and  $\{x_l\} = \frac{i}{p}+\epsilon_l$ where $0 < \epsilon_k, \epsilon_l < \frac{1}{p}$. Then, 
\begin{align*}
    |\epsilon_k - \epsilon_l| &= \left|\left(x_k-\floor{x_k} - \frac{i}{p}\right) - \left(x_l-\floor{x_l}-\frac{i}{p}\right)\right| \\
    &= |2(k-l)q/p - (\floor{x_k} - \floor{x_l})|
\end{align*}
and so $p|\epsilon_k - \epsilon_l| \in \mathbb{Z}$, making $\epsilon_k = \epsilon_l$ and thus $\{x_k\} = \{x_\ell\}$ by Lemma 1, a contradiction.

If instead $\floor{x_k} \not\equiv \floor{x_l} \pmod{2}$ such that $\{x_k\} = \frac{i}{p} + \epsilon_k$ and $1-\{x_l\} = \frac{i}{p}+\epsilon_l$, where $0 < \epsilon_k, \epsilon_l < \frac{1}{p}$, then consider
    \begin{align*}
    \epsilon_k + \epsilon_l &=\{x_k\} + (1-\{x_l\}) -2/p\\ & =(x_k -\floor{x_k}) - (x_l - \floor{x_l})+1 -2/p \\
    &= \frac{2(k-l)q}{p}  - \left(\floor{x_k} - \floor{x_l}+1\right) -2/p\
    \end{align*}
Then, $p(\epsilon_k+\epsilon_l)$ is an even integer, forcing $\epsilon_k = \epsilon_l = 0$, again a contradiction.

Therefore, since there are $p$ bouncing points, there must be exactly one bouncing point in each subinterval $\left( \frac{i}{p}, \frac{i+1}{p} \right) \times \{0\}$.
\end{proof}

Next, we show that the bouncing points are aligned in a specifically symmetric way. We say a bouncing point $\pi(x_k,k)$ is even if $\floor{x_k}$ is even and say it is odd otherwise.

\begin{lemma}\label{lem: alternate}
With the bottom side of $\mathbf{T}$ divided into $p$ subintervals of equal length, the following statements hold about bouncing points.

\begin{enumerate}
    \item Bouncing points in adjacent subintervals  cannot be of the same parity. That is, if $\pi(x_k, 2k) \in \left(\frac{i-1}{p}, \frac{i}{p}\right) \times \{0\}$ and $\pi(x_l, 2l) \in \left(\frac{i}{p}, \frac{i+1}{p}\right) \times \{0\}$, then $\floor{x_k} \not \equiv \floor{x_l} \pmod{2}$.
    \item  Two bouncing points in  adjacent subintervals of the bottom side are symmetric about the perpendicular bisector $x =\frac{i}{p}$ for some $i \in \{1, \ldots, p-1\}$
\end{enumerate}
\end{lemma}
\begin{proof}
     For $k, l\in \{1,\ldots, p-1\}$ such that $\floor{x_k} \equiv \floor{x_l} \pmod{2}$, we compute 
     \begin{align*}
         p\cdot d(\pi(x_k, 2k), \pi(x_l,2l))&= p\cdot|\{x_k\} -\{x_l\}| \\ 
         &= |2(k-l)q - p(\floor{x_k}-\floor{x_l})|.
     \end{align*}
The right hand side is always even, so the distance between different bouncing points of the same parity must be at least $2/p$. In particular, two even bouncing points (or two odd bouncing points) cannot be in the same subinterval. We conclude that bouncing points in adjacent subintervals alternate in parity. This proves (1).

For (2), without loss of generality, assume $\floor{x_k}$ is even and $\floor{x_l}$ is odd. Then,
    \begin{align*}
         |\{x_k\} + (1-\{x_l\})|&= |(x_k-x_l) -(\floor{x_k}-\floor{x_l}+1)| \\
         &= |2(k-l)q/p-(\floor{x_k}-\floor{x_l}+1)|.
    \end{align*}
The right hand side is an even integer, meaning the midpoint $\frac{1}{2}({x_k}+(1-{x_l}))$ of bouncing points with different parities belongs to $\{\frac{1}{p},\frac{2}{p}, \ldots, \frac{p-1}{p}\}$. In particular, this is true for the midpoint of two bouncing points on consecutive subintervals.
\end{proof}

We are now ready to prove Proposition \ref{prop: symmetry}.

\begin{proof}[Proof of Proposition \ref{prop: symmetry}]

By symmetry of the arguments in Lemmas \ref{lem: different}, \ref{lem: one-in-sub}, and \ref{lem: alternate}, analogous statements hold for bouncing points on the remaining sides of $\mathbf{T}$. We enumerate the bouncing points of $\gamma = \gamma(a,\frac{p}{q})$ along the bottom, top, left, and right sides of $\mathbf{T}$, respectively, as follows:
\begin{align*}
    \mathcal{C}_B &= \{ (h_1,0), (h_2,0), \ldots, (h_{p},0) \},\\
    \mathcal{C}_T &= \{ (h'_1,1), (h'_2, 1), \ldots, (h'_p,1) \},\\
    \mathcal{C}_L &= \{ (0, v_1), (0,v_2), \ldots, (0,v_q) \}, \text{ and}  \\
    \mathcal{C}_R &= \{ (1, v'_1), (1,v'_2), \ldots,(1,v'_q) \}  
\end{align*}
so that 
\begin{itemize}
\item $a = h_{i_0}$, where $i_0 = \lceil{pa}\rceil$,
\item $\frac{i-1}{p} < h_i, h'_i< \frac{i}{p}$ and $\frac{j-1}{q} < v_j, v'_j< \frac{j}{q}$, 
\item $\frac{1}{2}(h_{i}+h_{i+1}) =\frac{i}{p}= \frac{1}{2}(h'_i+h'_{i+1})$ for $i = 1, \ldots, p-1$, and $\frac{1}{2}(v_j+v_{j+1}) =\frac{j}{q}= \frac{1}{2}(v'_j+v'_{j+1})$ for $j = 1, \ldots, q-1$.
\end{itemize}

Every segment of $\gamma$ in $\mathbf{T}$ has slope $\frac{p}{q}$ or $-\frac{p}{q}$, so any two segments either intersect or are parallel. At each bouncing point on $\partial \mathbf{T}$, exactly two segments  meet there. We label the segments as follows: $\lambda_{B,i,+}$ denotes the segment meeting the bottom side at $(h_i,0)$ with slope $+\frac{p}{q}$, and similarly for the other sides using subscripts T, L, and R; see Figure \ref{fig:bouncing-symmetric}. Note that each segment carries two labels, one for each of its endpoints. 

\begin{figure}[h]
    \centering
    \includegraphics[width=0.65\linewidth]{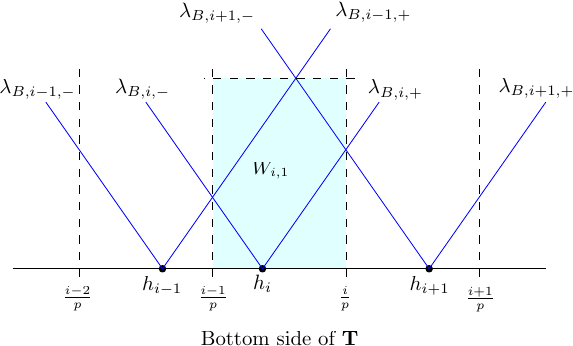}
    \caption{The rectangle $W_{i,1}$ along with some labelled segments of the trajectory $\gamma$}.
    \label{fig:bouncing-symmetric}
\end{figure}

Now consider the rectangle
$$W_{i,1} = \left[\frac{i-1}{p},\frac{i}{p} \right] \times \left[ 0,\frac{1}{q} \right],$$
where $i\in \{1, \ldots, p \}$. For now, let's suppose $i \neq 1,p$. Since $\lambda_{B,i-1,+}$ and $\lambda_{B,i,-}$ have opposite slopes and $x = \frac{i-1}{p}$ is the midpoint of $h_{i-1}$ and $h_i$, these two segments are symmetric about $x = \frac{i-1}{p}$ and meet at a point on that line. Similarly, $\lambda_{B,i,+}$ and $\lambda_{B,i+1,-}$ meet on $x = \frac{i}{p}$. Finally, $\lambda_{B,i-1,+}$ and $\lambda_{B,i+1,-}$ intersect at 
$$x = \frac{h_{i-1}+h_{i+1}}{2}, \qquad y = \frac{p}{q} \cdot \frac{h_{i+1}-h_{i-1}}{2} = \frac{p}{q} \cdot \frac{1}{p} = \frac{1}{q},$$
where the last equality uses the symmetry condition $h_{i+1} - h_{i-1} = \frac{2}{p}$. Hence this intersection lies on the top side of $W_{i,1}$.

Since the slopes of these segments are $\pm \frac{p}{q}$, there are no other intersection points of any segment of $\gamma$ with $\partial W_{i, 1}$, for otherwise there would be more bouncing points on the sides of $\mathbf{T}$ than those listed above. Therefore $W_{i,1} \cap \pi(\ell)$ is a parallelogram formed by two pairs of parallel segments connecting the four sides of $W_{i,1}$.

Moreover, since $\lambda_{B,i,-}$ and $\lambda_{B,i-1,+}$ form an isosceles triangle with base on $y = \frac{1}{q}$, the parallelograms in $W_{i-1,1}$ and $W_{i,1}$ are symmetric about $x = \frac{i-1}{p}$. Similarly, the parallelograms in $W_{i,1}$ and $W_{i+1,1}$ are symmetric about $x = \frac{i}{p}$.

For the boundary cases $i=1$ and $i=p$, the segments $\lambda_{B,0,+}$ and 
$\lambda_{B,p+1,-}$ are not defined. Instead, we use $\lambda_{L,1,+}$ and 
$\lambda_{R,p,-}$ in their respective roles. Since $\lambda_{L,1,+}$ passes 
through both $(0,v_1)$ and $(h_1,0)$, we compute $v_1 = \frac{p}{q}h_1$, and 
similarly, since $\lambda_{R,p,-}$ passes through both $(1,v'_1)$ and $(h_p,0)$, 
we get $v'_1 = \frac{p}{q}(1-h_p)$. The same argument then applies to conclude 
that $\gamma$ intersects $\partial W_{1,1}$ and $\partial W_{p,1}$ in a parallelogram.

We then proceed by induction on the row index $j$. The base case $j = 1$ is proved above. Observe that by the parallelogram structure of row $j$, the intersections of 
$\pi(\ell)$ with $y = \frac{j}{q}$ satisfy exactly one point per subinterval 
$\left(\frac{i-1}{p}, \frac{i}{p}\right)$, symmetric about each $x = \frac{i}{p}$, 
that is, the symmetry conditions of Lemmas \ref{lem: one-in-sub} and 
\ref{lem: alternate}. So we may treat them as the bottom bouncing points. The same argument therefore applies to the rectangles $W_{i,j+1}$, giving a parallelogram in each and the desired symmetry about $y = \frac{j}{q}$. The proposition now follows by induction.
\end{proof}

Proposition \ref{prop: symmetry} allows us to draw out the entire orbit without having to trace the trajectory. Indeed, visualizing the orbit of $\gamma(a, p/q)$ is now a straightforward task using the following steps:
\begin{enumerate}
    \item Divide each horizontal side of $\mathbf{T}$ into $p$ equal subintervals and each vertical side into $q$ equal subintervals. This partitions $\mathbf{T}$ into $pq$ isometric rectangles, namely the $W_{i,j}$'s, each of width $1/p$ and height $1/q$, for $1 \leq i \leq p$ and $1 \leq j \leq q$.
    \item Locate the point $(a,0)$ and draw the parallelogram in $W_{i_0,1}$, 
    where $i_0 = \lceil pa \rceil$, with one vertex at $(a,0)$ and sides of 
    slope $\pm\frac{p}{q}$.
    \item Iteratively reflect this parallelogram about the vertical lines $x = 
    \frac{i}{p}$ and horizontal lines $y = \frac{j}{q}$ until every $W_{i,j}$ 
    contains a parallelogram. The resulting figure is the desired orbit.
\end{enumerate}

\section{Optimal Covering Radius}
Unlike a non-periodic orbit, which has dense orbit in $\mathbf{T}$, a periodic 
orbit is not dense. To answer our two problems, we will need to compute $\rcov(\gamma)$ given the combinatorial data of $\gamma$. 

For the simple case when $\gamma$ is a period-2 trajectory, we have that $\rcov(\gamma)$ is the distance from the orbit to the furthest side of $\mathbf{T}$ that is parallel to $\gamma$.

From now, we will consider trajectories with period at least 4. By the orbit symmetry in Proposition~\ref{prop: symmetry}, it suffices to compute the optimal covering 
radius of a parallelogram inscribed in a single subrectangle $W_{i,j}$, since 
all such subrectangles are isometric.

\begin{lemma}\label{lem: parallelogram} 
Let $ABCD$ be a rectangle and let $PQRS$ be a parallelogram inscribed in $ABCD$ 
with $P$ on side $AB$, such that $\angle QPB = \theta$, as in Figure~\ref{fig: 
inscribed-rectangle}. Then the optimal radius $r$ such that the $r$-neighborhood 
of $PQRS$ covers $ABCD$ is
$$r = \max \{ |AP|\sin \theta,\, |BP| \sin\theta \}.$$
\end{lemma}

\begin{figure}[h]
    \centering
    \includegraphics[scale=0.7]{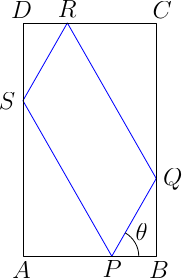}
    \caption{A rectangle $ABCD$ with an inscribed parallelogram $PQRS$}
    \label{fig: inscribed-rectangle}
\end{figure}

\begin{proof} Note that $\angle SPA = \angle DRS =\angle CRQ =\angle QPB = \theta$. 
To cover the rectangle, the radius must be at least the maximum of the distances from the vertices of $ABCD$ to the nearest side of $PQRS$. We compute
\begin{align*}
    d(A, PS) &= |AP|\sin\theta = |RC|\sin\theta =  d(C, RQ), \\ 
    d(B, PQ) &= |BP|\sin\theta = |DR|\sin\theta=  d(D, RS).
\end{align*}
It then remains to compare these with half the perpendicular distance between each 
pair of opposite sides of $PQRS$. Since $|PB| = |PQ|\cos\theta$, we have
\begin{align*}
    d(PS, QR) &= |PQ|\sin(\pi-2\theta) \\
                &= |PQ|\sin2\theta, \\
                &= 2(|PQ|\cos\theta)\sin\theta \\
                &= 2|PB|\sin\theta.
\end{align*}
Therefore,
$$\frac{1}{2}d(PS,QR) = |PB|\sin\theta \leq \max\{|AP|\sin\theta, |PB|\sin\theta\}$$
and similarly $\frac{1}{2}d(PQ,RS) = |AP|\sin\theta \leq \max\{|AP|\sin\theta, 
|PB|\sin\theta\}$. The corner distances therefore dominate the half-gaps between 
opposite sides. Clearly, the maximum of these distances is the optimal covering radius, that is  
$$r = \max\{|AP|\sin\theta,\, |BP|\sin\theta\}$$
as claimed.
\end{proof}

Given $a \in (0,1)$ and relatively prime positive integers $p,q$, by 
Proposition~\ref{prop: symmetry}, the orbit of $\gamma(a,\frac{p}{q})$ 
intersects each subrectangle $W_{i,j}$ in a parallelogram with the two corner distances being $\frac{\{pa\}}{p}$ and $\frac{1-\{pa\}}{p}$. Applying 
Lemma~\ref{lem: parallelogram} and noting that $\sin\theta = \frac{p}{\sqrt{p^2+q^2}}$,  
we obtain the answer for \ref{prob1} as follows.
\begin{theorem}\label{thm: covering radius}
Let $\mathcal{C}_N(p)$ denote the class of non-singular periodic trajectories of period $N$ 
with $p$ bouncing points on the bottom edge and $q$ bouncing points on the left edge so that $2(p+q) = N$. For $\gamma = \gamma(a,\frac{p}{q}) 
\in \mathcal{C}_N(p)$,
\begin{equation}
    \rcov(\gamma(a,\tfrac{p}{q})) = \frac{1}{\sqrt{p^2+q^2}}\max(\{pa\},\, 1-\{pa\}).
\end{equation}
Moreover, it follows that in the class $\mathcal{C}_N(p)$, the covering 
radius is minimized when $a \in \left\{\frac{1}{2p}, \frac{3}{2p}, \ldots, 
\frac{2p-1}{2p}\right\}$, giving
\begin{equation}\label{eq: min-rcov}
    \min_{\gamma \in \mathcal{C}_N(p)} \rcov(\gamma) = \frac{1}{2\sqrt{p^2+q^2}}.
\end{equation}
These radius-minimizing trajectories correspond to those starting at the midpoint of each of the subintervals $\left[\frac{i}{p}, \frac{i+1}{p}\right] \times \{0\}$ on the bottom edge.

\end{theorem}

\begin{remark}

For readers familiar with the torus, by identifying opposite sides of the $2 \times 2$ square, a billiard path in the square lifts to a geodesic on the torus $\mathbb{R}^2/2\mathbb{Z}^2$. Under this identification, the optimal covering radius $\rcov(\gamma)$ corresponds to the smallest $r$ such that the $2r$-neighborhood of the geodesic covers the torus. Our results therefore give explicit covering radii for periodic geodesics on this flat torus.

\end{remark}

Our answer to \ref{prob2} is the following theorem.

\begin{theorem}
Let $r>0$ and let $M$ be the smallest positive integer greater than $\dfrac{1}{4r^2}$ that can be written as sum of two squares of relatively prime integers. Then, any shortest periodic path $\gamma$ such that $N_{r}(\gamma)$ covers $\mathbf{T}$ has the form $\gamma(a,\frac{p}{q})$ where $p^2+q^2 = M$, $\gcd(p,q) = 1$, and $a \in \left\{\frac{1}{2p}, \frac{3}{2p}, \ldots, 
\frac{2p-1}{2p}\right\}$.
\end{theorem}

\begin{proof}
For a given radius $r>0$, we seek the shortest periodic trajectory $\gamma$ such that $N_r(\gamma) \supset \mathbf{T}$. By the above, we need relatively prime positive integers $p$ and $q$ satisfying
$$r \geq \rcov(\gamma) \geq \frac{1}{2\sqrt{p^2+q^2}}$$
with $2\sqrt{p^2+q^2}$ minimal. This is equivalent to finding the smallest integer greater than $\dfrac{1}{4r^2}$ expressible  as a sum of two squares of relatively prime integers. The claim about the starting point $a$ follows from the observation preceding Equation \ref{eq: min-rcov} in Theorem \ref{thm: covering radius}.
\end{proof} 

This sum of squares criterion is easy to verify in practice. A classical theorem from number theory states that $M$ is 
properly represented as a sum of two squares of coprime integers if and only if $M$ has no prime factor congruent to $3 \pmod{4}$ and is not divisible by $4$, see \cite[Chapter VI.5]{Davenport}. To find such $p, q$ explicitly, one may apply Cornacchia's algorithm; see \cite[Algorithm~1.5.2]{Cohen} and also  \cite{Basilla}.

\end{document}